\documentclass[11pt]{article}
\usepackage[T1]{fontenc}
\usepackage{amssymb}
\usepackage{amsmath,amsthm}
\usepackage{amsfonts}
\usepackage{graphicx}
\usepackage{tikz}
\usepackage{flafter}

\newtheorem{theorem}{Theorem}

\newtheorem{definition}[theorem]{Definition}
\newtheorem{proposition}[theorem]{Proposition}
\newtheorem{corollary}[theorem]{Corollary}
\newtheorem{lemma}[theorem]{Lemma}
\newtheorem{remark}[theorem]{Remark}

\newtheorem{example}[theorem]{Example}

\begin{document}

\begin{center}
\textbf{On Nilcompactifications of Prime Spectra of Commutative Rings}

\bigskip

Lorenzo Acosta G.\footnote{%
Mathematics Department, Universidad Nacional de Colombia, AK 30 45-03, Bogot%
\'{a}, Colombia. e-mail: lmacostag@unal.edu.co} and I. Marcela Rubio P.%
\footnote{%
Corresponding author. Mathematics Department, Universidad Nacional de
Colombia, AK 30 45-03, Bogot\'{a}, Colombia. e-mail: imrubiop@unal.edu.co}

\bigskip
\end{center}

\begin{quote}
\textbf{Abstract: }Given a ring $R$ and $S$ one of its proper ideals, we obtain a
compactification of the prime spectrum of $S$ through a mainly algebraic
process. We name it the $R-$nilcompactification of $\mathrm{Spec}S.$ We
study some categorical properties of this construction.

\textbf{Keywords: }Prime ideal, prime spectrum, spectral compactness,
compactification.

\textbf{MSC: }54D35, 54B30, 54B35.
\end{quote}

\bigskip

\section{Introduction.}

Compactification of a topological space is an important topic considered in a wide range of branches in mathematics as it guarantees a useful property for the space, see for example \cite{Arhangel}, \cite{Magill}, \cite{Margalef}, \cite{Morrison}, \cite{Nakassis}, \cite{Salbany}, \cite{Zuniga}. In this paper we study, from a categorical point of view, the compactification method of prime spectra which is presented in \cite{Acosta-Rubio}, named nilcompactification. Nilcompactification is a topological method obtained mainly through an algebraic process. The categorical point of view of nilcompactification, with some of its possible variations, offers an interesting wealth for this process. For the categorical concepts see, for example, \cite{MacLane}. This method of nilcompactification is functorial in a simple way in that it has interesting properties and the involved constructions provide us with different natural transformations. We can take into account that some compactification processes need to consider suitable subcategories to obtain a functorial behavior, as studied in \cite{Acosta-Giraldo} for the Alexandroff compactifications.

In this paper the word \textit{ring} means \textit{commutative ring, not
necessarily with identity}. A homomorphism is a function between rings that
respects addition and product. We suppose that prime ideals are proper ideals by definition. The prime spectrum of a ring $S$ is denoted
by $\mathrm{Spec}S$ and it is the set of prime ideals of $S$ endowed with
the Zariski topology. In this topology the sets 
\begin{equation*}
D_{S}(a)=\{P\in \mathrm{Spec}S:a\notin P\},
\end{equation*}%
where $a\in S,$ provide a basis. The closed sets are%
\begin{equation*}
V_{S}(I)=\{P\in \mathrm{Spec}S:P\supseteq I\},
\end{equation*}%
where $I$ is an ideal of $S.$ It is known that the basic open sets are
compact and for each ideal $I$ of $S$ the function $V_{S}(I)\rightarrow $ $%
\mathrm{Spec}\left( S/I\right) :P\mapsto P/I$ is a homeomorphism (see \cite%
{Atiyah}). The Zariski topology is also called the hull-kernel topology
because the closure of a subset $\mathcal{B}$ of $\mathrm{Spec}S$ is the set 
\begin{equation*}
\left\{ P\in \mathrm{Spec}S:P\supseteq \bigcap\limits_{J\in \mathcal{B}%
}J\right\}
\end{equation*}%
and $\bigcap\limits_{J\in \mathcal{B}}J$ is called the kernel of $\mathcal{B}%
.$

The nilradical of $S,$ denoted $N\left( S\right) ,$ is precisely the kernel
of $\mathrm{Spec}S.$

It is known that $S$ is semiprime or reduced if $N\left( S\right) =\left\{
0\right\} .$

A ring whose spectrum is compact is a \textit{spectrally compact ring. }In
particular, every unitary ring $R$ is spectrally compact because $D_{R}(1)=%
\mathrm{Spec}R.$

We use the following notations:

\noindent $\mathcal{CR}:$ Category of commutative rings and homomorphisms of
rings.

\noindent $\mathcal{CR}_{1}:$ Category of unitary commutative rings and
homomorphisms of unitary rings.

\noindent $\mathcal{CR}^{s}:$ Category of commutative rings and surjective
homomorphisms.

\noindent $\mathrm{Top}:$ Category of topological spaces and continuous
functions.

\noindent $\mathbb{S}:$ Category of spectral spaces and strongly continuous
functions. A spectral space is a topological space that is homeomorphic to
the prime spectrum of a unitary commutative ring. It is known that a
topological space is spectral if and only if it is sober, compact and
coherent (see \cite{Hochster}). A function is strongly continuous if it
sends compact open sets in compact open sets by reciprocal image.

\section{The mechanism of nilcompactification}

The material of this section is taken from \cite{Acosta-Rubio}.

In this section $S$ is a fixed ring and $R$ is an i-extension of $S,$ that
is, a ring containing $S$ as ideal.

Given an ideal $I$ of $S$ it is clear that the set $\psi (I)=\{x\in
R:xS\subseteq I\}$ is an ideal of $R.$ So, $\psi $ is a function from the
set $\mathcal{J}\left( S\right) $ of the ideals of $S,$ to the set $\mathcal{%
J}(R)$ of the ideals of $R.$

\begin{lemma}
The function $\psi :\mathcal{J}(S)\rightarrow \mathcal{J}(R)$ has the
following properties:

(i) If $P$ is a prime ideal of $S$ then $\psi (P)$ is a prime ideal of $R$
not containing $S.$

(ii) If $P$ and $Q$ are prime ideals of $S$ such that $\psi (P)=\psi (Q)$
then $P=Q.$

(iii) If $Q$ is a prime ideal of $R$ not containing $S$ then $Q\cap S$ is a
prime ideal of $S$ and $\psi (Q\cap S)=Q.$

(iv) $\psi \left( \bigcap\limits_{P\in \mathrm{Spec}S}P\right)
=\bigcap\limits_{P\in \mathrm{Spec}S}\psi (P).$
\end{lemma}

\begin{proposition}
The function $\psi :\mathrm{Spec}S\rightarrow \mathrm{Spec}R$ is injective,
continuous and open onto its image.
\end{proposition}

\begin{proof}
By (iii) of the previous lemma we have that $\psi $ is injective.

For continuity, it is enough to observe that if $r\in R$ then $\psi
^{-1}\left( D_{R}\left( r\right) \right) =\bigcup\limits_{s\in
S}D_{S}\left( rs\right) .$

On the other hand, if $s\in S$ then it is easy to see that $\psi \left(
D_{S}\left( s\right) \right) =\psi \left( \mathrm{Spec}S\right) \cap
D_{R}\left( s\right) ,$ then $\psi $ is open onto its image.
\end{proof}

Hereinafter we denote with $\mathrm{Spec}_{S}R$ the image of the function $%
\psi $ restricted to $\mathrm{Spec}S.$ In other words, $\mathrm{Spec}_{S}R$
is the set $\{Q\in \mathrm{Spec}R:Q\nsupseteq S\}.$ Thus, $\mathrm{Spec}S$
is homeomorphic to $\mathrm{Spec}_{S}R,$ seen as subspace of $\mathrm{Spec}%
R. $

\begin{proposition}
$\mathrm{Spec}_{S}R$ is an open of $\mathrm{Spec}R$ and its closure is a
subspace of $\mathrm{Spec}R$ homeomorphic to $\mathrm{Spec}\left( R/\psi
\left( N\left( S\right) \right) \right) .$
\end{proposition}

\begin{proof}
The first assertion follows from the equality $\mathrm{Spec}_{S}R=\bigcup%
\limits_{s\in S}D_{R}\left( s\right) .$ On the other hand:%
\begin{eqnarray*}
\overline{\mathrm{Spec}_{S}R} &=&\left\{ Q\in \mathrm{Spec}R:Q\supseteq
\bigcap\limits_{P\in \mathrm{Spec}S}\psi (P)\right\}  \\
&=&\left\{ Q\in \mathrm{Spec}R:Q\supseteq \psi \left( \bigcap\limits_{P\in 
\mathrm{Spec}S}P\right) \right\}  \\
&=&\{Q\in \mathrm{Spec}R:Q\supseteq \psi \left( N\left( S\right) \right) \}
\\
&=&V_{R}\left( \psi \left( N\left( S\right) \right) \right)  \\
&\approx &\mathrm{Spec}\left( R/\psi \left( N\left( S\right) \right) \right)
.
\end{eqnarray*}
\end{proof}

\begin{remark}
Notice that the inclusion of $\mathrm{Spec}S$ in $\mathrm{Spec}\left( R/\psi
\left( N\left( S\right) \right) \right) $ is given by the function 
\begin{equation*}
\lambda :\mathrm{Spec}S\rightarrow \mathrm{Spec}\left( R/\psi \left( N\left(
S\right) \right) \right) :P\mapsto \psi (P)/\psi \left( N\left( S\right)
\right) .
\end{equation*}
\end{remark}

\begin{corollary}
If $R$ is spectrally compact (in particular if $R$ has identity) then $%
\mathrm{Spec}\left( R/\psi \left( N\left( S\right) \right) \right) $ is a compactification of $\mathrm{Spec}S$ in which $\mathrm{Spec}S$ is open.
\end{corollary}

\begin{definition}
If $R$ is spectrally compact, the space $\mathrm{Spec}\left( R/\psi \left(
N\left( S\right) \right) \right) $ is called the $R$-nilcompactification of $%
\mathrm{Spec}S.$
\end{definition}

It is well known that every ring is an ideal of a unitary ring (see \cite%
{Jacobson - Book}), therefore we obtain the following result:

\begin{theorem}
The spectrum of every ring has a spectral compactification.
\end{theorem}

\begin{proof}
It is enough to observe that if $S$ is a ring, we can choose $R$ unitary
and hence $\mathrm{Spec}\left( R/\psi \left( N\left(
S\right) \right) \right) $ is a spectral space.
\end{proof}

\section{Functorial behavior of the mechanism of nilcompactification}

Consider the category $\mathcal{E}$ whose objects are the pairs $\left(
S,R\right) $ with $R$ a unitary i-extension of $S$ and where the morphisms
from $\left( S_{1},R_{1}\right) $ to $\left( S_{2},R_{2}\right) $ are the
homomorphisms of unitary rings from $R_{1}$ to $R_{2}$ such that $%
h(S_{1})=S_{2}.$

As $S$ and $R$ are variables, in this context, the functions $\psi $ and $%
\lambda $ defined in the previous section will be denoted $\psi _{\left(
S,R\right) }$ and $\lambda _{(S,R)}$ respectively$.$

For each object $(S,R)$ of $\mathcal{E},$ we define $Q(S,R)=R/\psi _{\left(
S,R\right) }\left( N(S)\right) .$

The following proposition can be proved without difficulty.

\begin{proposition}
\label{Q well defined}If $h:(S,R)\rightarrow (T,M)$ is a morphism of $%
\mathcal{E}$, then 
\begin{equation*}
Q(h):Q(S,R)\rightarrow Q(T,M):r+\psi _{\left( S,R\right) }\left( N(S)\right)
\mapsto h(r)+\psi _{\left( T,M\right) }\left( N(T)\right)
\end{equation*}%
is well defined and is a homomorphism of unitary rings.
\end{proposition}

Thus, $Q$ is a functor from the category $\mathcal{E}$ to the category $%
\mathcal{CR}_{1}$ of unitary commutative rings.

We denote $\mathrm{NC}$ the contravariant functor $\mathrm{Spec}\circ Q:%
\mathcal{E}\rightarrow \mathbb{S}$ where $\mathbb{S}$ is the category of the
spectral spaces and the strongly continuous functions. So, $\mathrm{NC}(S,R)$
is the $R$-nilcompactification of $\mathrm{Spec}S.$

\bigskip

\textbf{Some natural transformations:}

Consider the functors $V:\mathcal{E}\rightarrow \mathcal{CR}_{1}$ and $W:%
\mathcal{E}\rightarrow \mathcal{CR}^{s}$ defined as follows:%
\begin{eqnarray*}
V(h &:&(S,R)\rightarrow (T,M))=h:R\rightarrow M\text{ \ and} \\
W(h &:&(S,R)\rightarrow (T,M))=h\mid _{S}:S\rightarrow T.
\end{eqnarray*}

The following proposition is a direct consequence of the Proposition \ref{Q
well defined}:

\begin{proposition}
If for each object $(S,R)$ of $\mathcal{E}$, we denote by $\theta
_{(S,R)}:R\rightarrow Q(S,R)$ the canonical function to the quotient, then $%
\theta =\left( \theta _{\left( S,R\right) }\right) _{\left( S,R\right) \in Ob%
\mathcal{E}}$ is a natural transformation from the functor $V$ to the
functor $Q$.
\end{proposition}

\begin{figure}[!htb]
\centering
\begin{tikzpicture}[scale=0.8]
	\draw[->] (0,-0.3) --node[left]{\footnotesize $h$} (0,-1.7);
	\draw[-] (0.8,1.5) -- (0.8,-2.7);
	\draw[-] (-1,0.9) -- (11,0.9);
	\draw[->] (2.5,-0.3) --node[left]{\footnotesize $V(h)$} (2.5,-1.7);
	\draw[->] (6.5,-0.3) --node[right]{\footnotesize $Q(h)$} (6.5,-1.7);
    \draw[->] (3.8,0) --node[above]{\scriptsize $\theta_{(S,R)}$} (5,0);
    \draw[->] (3.8,-2.2) --node[below]{\scriptsize $\theta_{(T,M)}$} (5,-2.2);
	\draw (0,1.3) node{\small $\mathcal{E}$} (4.4,1.3) node{\small $\mathcal{CR}_{1}$};
	\draw (0,0) node{\small $(S,R)$} (0,-2.2) node{\small $(T,M)$};
	\draw (2.3,0) node{\small $V(S,R)=R$} (2.3,-2.2) node{\small $V(T,M)=M$};
	\draw (8,0) node{\small $Q(S,R)=R/\psi _{(S,R)}(N(S))$} (8,-2.2) node{\small $Q(T,M)=M/\psi _{(T,M)}(N(T))$};
\end{tikzpicture}
\end{figure}

\begin{proposition}
$\psi =\left( \psi _{(S,R)}\right) _{(S,R)\in Ob\mathcal{E}}$ is a natural
transformation from the functor $\mathrm{Spec}\circ W$ to the functor $%
\mathrm{Spec}\circ V.$
\end{proposition}

\begin{proof}
It is enough to observe that if $h:(S,R)\rightarrow (T,M)$ is a morphism of $\mathcal{%
E}$ then for each prime ideal $P$ of $T$ we have that $h^{-1}\left(
\psi _{(T,M)}\left( P\right) \right) =\psi _{(S,R)}\left( h^{-1}\left(
P\right) \right) .$
\end{proof}

\begin{proposition}
\label{lambda inclusion in NC}$\lambda =\left( \lambda _{(S,R)}\right)
_{(S,R)\in Ob\mathcal{E}}$ is a natural transformation from the functor $%
\mathrm{Spec}\circ W$ to the functor $\mathrm{NC}.$%
\begin{figure}[!htb]
\centering
\begin{tikzpicture}[scale=0.8]
	\draw[->] (0,-0.3) --node[left]{\footnotesize $h$} (0,-1.7);
	\draw[-] (0.8,1.5) -- (0.8,-2.7);
	\draw[-] (-1,0.9) -- (8.5,0.9);
	\draw[<-] (3.5,-0.3) --node[left]{\footnotesize $\mathrm{Spec}\circ W(h)$} (3.5,-1.7);
	\draw[<-] (6.5,-0.3) --node[right]{\footnotesize $\mathrm{NC}(h)$} (6.5,-1.7);
    \draw[->] (4.2,0) --node[above]{\scriptsize $\lambda _{(S,R)}$} (6,0);
    \draw[->] (4.2,-2.2) --node[below]{\scriptsize $\lambda _{(T,M)}$} (6,-2.2);
	\draw (0,1.3) node{\small $\mathcal{E}$} (4.9,1.3) node{\small $\mathrm{Top}$};
	\draw (0,0) node{\small $(S,R)$} (0,-2.2) node{\small $(T,M)$};
	\draw (3.3,0) node{\small $\mathrm{Spec}(S)$} (3.3,-2.2) node{\small $\mathrm{Spec}(T)$};
	\draw (7.1,0) node{\small $\mathrm{NC}(S,R)$} (7.1,-2.2) node{\small $\mathrm{NC}(T,M)$};
\end{tikzpicture}
\end{figure}
\end{proposition}

\begin{proof}
Let $h:(S,R)\rightarrow (T,M)$ be a morphism of $\mathcal{E}$, $P\in \mathrm{%
Spec}T$ and $r\in R.$%
\begin{eqnarray*}
\begin{array}{ll}
r+\psi _{(S,R)}\left( N(S)\right)  &\in Q(h)^{-1}\left( \psi _{(T,M)}\left(
P\right) /\psi _{(T,M)}\left( N(S)\right) \right)  \\
&\Leftrightarrow h(r)\in \psi _{(T,M)}\left( P\right)  \\
&\Leftrightarrow h(r)T\subseteq P \\
&\Leftrightarrow h(rS)\subseteq P \\
&\Leftrightarrow rS\subseteq h^{-1}(P) \\
&\Leftrightarrow r\in \psi _{(S,R)}\left( h^{-1}(P)\right)  \\
&\Leftrightarrow r+\psi _{(S,R)}\left( N(S)\right) \in \psi _{(S,R)}\left(
h^{-1}(P)\right) /\psi _{(S,R)}\left( N(S)\right) 
\end{array}
\end{eqnarray*}

Thus, $NC(h)\left( \lambda _{(T,M)}\left( P\right) \right) =\lambda
_{(S,R)}\left( \left( \mathrm{Spec}\circ W\right) \left( h\right) \left(
P\right) \right) .$
\end{proof}

\begin{proposition}
Let $h:(S,R)\rightarrow (T,M)$ be a morphism of the category $\mathcal{E}.$
The function $h$ can be restricted to a function from $\psi _{\left(
S,R\right) }\left( N(S)\right) $ to $\psi _{(T,M)}\left( N(T)\right) .$
\end{proposition}

\begin{proof}
Consider $r\in \psi _{\left( S,R\right) }\left( N(S)\right) ,$ that is, $rs\in
N(S)$ for all $s\in S.$ Given $t\in T,$ there exists $s\in S$ such that $h\left(
s\right) =t.$ Thus, $h\left( r\right) t=h\left( r\right) h\left(
s\right) =h\left( rs\right) \in h\left( N\left( S\right) \right) \subseteq
N\left( T\right) $ so that $h\left( r\right) \in \psi _{(T,M)}\left(
N(T)\right) .$
\end{proof}

Let $\chi :\mathcal{E}\rightarrow \mathcal{CR}$ be the functor defined by: 
\begin{equation*}
\chi \left( h:(S,R)\rightarrow (T,M)\right) =h\mid _{\psi _{\left(
S,R\right) }\left( N(S)\right) }:\psi _{\left( S,R\right) }\left(
N(S)\right) \rightarrow \psi _{(T,M)}\left( N(T)\right)
\end{equation*}%
and for the object $(S,R)$ of $\mathcal{E}$, denote $j_{(S,R)}$ the natural
inclusion of $\psi _{\left( S,R\right) }\left( N(S)\right) $ into the ring $%
R.$

Notice that for each $h:(S,R)\rightarrow (T,M)$, morphism of the category $%
\mathcal{E}$, the square in the following figure is commutative; this allows
us to state the following result.

\begin{figure}[!htb]
\centering
\begin{tikzpicture}[scale=0.8]
	\draw[->] (0,-0.3) --node[left]{\footnotesize $h$} (0,-1.7);
	\draw[-] (0.8,1.5) -- (0.8,-2.7);
	\draw[-] (-1,0.9) -- (9,0.9);
	\draw[->] (3,-0.3) --node[left]{\footnotesize $\chi (h)$} (3,-1.7);
	\draw[->] (6.5,-0.3) --node[right]{\footnotesize $V(h)$} (6.5,-1.7);
    \draw[->] (4.3,0) --node[above]{\scriptsize $j_{(S,R)}$} (5.5,0);
    \draw[->] (4.3,-2.2) --node[below]{\scriptsize $j_{(T,M)}$} (5.5,-2.2);
	\draw (0,1.3) node{\small $\mathcal{E}$} (4.7,1.3) node{\small $\mathcal{CR}$};
	\draw (0,0) node{\small $(S,R)$} (0,-2.2) node{\small $(T,M)$};
	\draw (2.8,0) node{\small $\psi _{(S,R)}(N(S))$} (2.8,-2.2) node{\small $\psi _{(T,M)}(N(T))$};
	\draw (7,0) node{\small $V(S,R)=R$} (7,-2.2) node{\small $V(T,M)=M$};
\end{tikzpicture}
\end{figure}

\begin{proposition}
$j=\left( j_{(S,R)}\right) _{(S,R)\in Ob\mathcal{E}}$ is a natural
transformation from the functor $\chi $ to the functor $V.$
\end{proposition}

\section{First variation}

Fix a ring $S.$ Let $\mathcal{E}\left( S\right) $ be the subcategory of $%
\mathcal{CR}_{1}$ whose objects are the i-extensions of $S$ and whose
morphisms are those that can be restricted to the identity of $S.$ It is
clear that $\mathcal{E}(S)$ can be identified with a subcategory of $%
\mathcal{E}$ and therefore the functor $Q$ can be restricted to a functor $%
Q_{S}:\mathcal{E}(S)\rightarrow \mathcal{CR}_{1}$ and the functor $\mathrm{NC%
}$ can be restricted to a functor $\mathrm{NC}_{S}:\mathcal{E}(S)\rightarrow 
\mathbb{S}.$

\begin{proposition}
If $h:R\rightarrow T$ is a morphism of $\mathcal{E}(S)$ then $%
Q_{S}(h):Q(S,R)\rightarrow Q(S,T)$ is injective.
\end{proposition}

\begin{proof}
\begin{eqnarray*}
r+\psi _{(S,R)}\left( N(S)\right) &\in &\ker Q_{S}(h)\Leftrightarrow h(r)\in
\psi _{(S,T)}\left( N(S)\right) \\
&\Leftrightarrow &h(r)s\in N(S)\text{ for all }s\in S \\
&\Leftrightarrow &h(rs)\in N(S)\text{ for all }s\in S \\
&\Leftrightarrow &rs\in N(S)\text{  for all }s\in S \\
&\Leftrightarrow &r\in \psi _{(S,R)}\left( N(S)\right) \\
&\Leftrightarrow &r+\psi _{(S,R)}\left( N(S)\right) =0
\end{eqnarray*}
\end{proof}

\begin{corollary}
\label{Surjec hom then NC homeomo}If $h:R\rightarrow T$ is a surjective
morphism of $\mathcal{E}(S)$ then $Q_{S}(h):Q(S,R)\rightarrow Q(S,T)$ is an
isomorphism and therefore $\mathrm{NC}_{S}(h):\mathrm{NC}_{S}\left( T\right)
\rightarrow \mathrm{NC}_{S}\left( R\right) $ is a homeomorphism.
\end{corollary}

\begin{corollary}
If $h:R\rightarrow T$ is a morphism of $\mathcal{E}(S)$ then $\mathrm{NC}%
_{S}(h)\left( \mathrm{NC}_{S}\left( T\right) \right) $ is dense in $\mathrm{%
NC}_{S}\left( R\right) .$
\end{corollary}

\begin{proof}
It is consequence of the injectivity of $Q_{S}(h):Q_{S}\left( R\right)
\rightarrow Q_{S}(T)$ (see \cite{Atiyah}).
\end{proof}

\begin{proposition}
If $h:R\rightarrow T$ is a morphism of $\mathcal{E}(S)$, then $\mathrm{NC}%
_{S}(h)\circ \lambda _{(S,T)}=\lambda _{(S,R)}.$%
\begin{figure}[tbh]
\centering
\begin{tikzpicture}[scale=0.8]
	\draw[<-] (6.5,-0.3) --node[right]{\footnotesize $\mathrm{NC}_{S}(h)$} (6.5,-1.7);
    \draw[->] (3.7,-1) --node[above]{\scriptsize $\lambda _{(S,R)}$} (6,-0.1);
    \draw[->] (3.7,-1.2) --node[below]{\scriptsize $\lambda _{(S,T)}$} (6,-2.2);
	\draw (2.8,-1.1) node{\small $\mathrm{Spec}(S)$};
	\draw (7.1,0) node{\small $\mathrm{NC}_{S}(R)$} (7.1,-2.2) node{\small $\mathrm{NC}_{S}(T)$};
\end{tikzpicture}
\end{figure}
\end{proposition}

\begin{proof}
It is a direct consequence of the Proposition \ref{lambda inclusion in NC}.
\end{proof}

\begin{corollary}
If $h:R\rightarrow T$ is a morphism of $\mathcal{E}(S)$ then $\mathrm{Spec}S$
is a subspace of $\mathrm{NC}_{S}(h)\left( \mathrm{NC}_{S}\left( T\right)
\right) .$
\end{corollary}

\section{Second variation}

Given a ring $S$, we denote $U_{0}\left( S\right) $ the set $S\times \mathbb{%
Z}$ endowed with the operations:%
\begin{eqnarray*}
(s,\alpha )+(t,\beta ) &=&(s+t,\alpha +\beta ) \\
(s,\alpha )(t,\beta ) &=&(st+\beta s+\alpha t,\alpha \beta ).
\end{eqnarray*}%
It is well known that $U_{0}\left( S\right) $ is a unitary ring and that, if
we identify $S$ with $S_{0}=S\times \{0\},$ $S$ is an ideal of $U_{0}(S).$
Besides we have the following universal property:

\begin{theorem}
Let $S$ be a ring and consider $\iota _{S}:S\rightarrow U_{0}(S):s\mapsto
(s,0).$ If $R$ is a unitary ring and $g:S\rightarrow R$ is a homomorphism,
then there exists a unique homomorphism of unitary rings $\widetilde{g}%
:U_{0}(S)\rightarrow R$ such that $\widetilde{g}\circ \iota _{S}=g.$%
\begin{figure}[tbh]
\centering
\begin{tikzpicture}[scale=0.8]
	\draw[->] (2.6,0) --node[above]{\footnotesize $\iota_{S}$} (4.4,0);
    \draw[->] (5,-0.3) --node[right]{\footnotesize $\exists !$ $\widetilde{g}$} (5,-1.8);
    \draw[->] (2.55,-0.3) --node[below]{\footnotesize $g$} (4.7,-1.9);
	\draw (2.3,0) node{\small $S$} (5.2,0) node{\small $U_{0}(S)$} (5,-2.1) node{\small $R$};
\end{tikzpicture}
\end{figure}
\end{theorem}

This property allows us to extend $U_{0}$ to a functor from the category $%
\mathcal{CR}$ of commutative rings to the category $\mathcal{CR}_{1}$ of
unitary rings, defining $U_{0}(h)=\widetilde{\iota _{T}\circ h}$, for a
homomorphism $h:S\rightarrow T.$ We have also that $\iota =(\iota
_{S})_{S\in Ob\mathcal{CR}}$ is a natural transformation from the identity
functor of the category $\mathcal{CR}$ to the functor $U_{0}$ considered as
endo-functor of $\mathcal{CR}.$

Notice that if $h:S\rightarrow T$ is a surjective homomorphism then $%
U_{0}(h) $ is a morphism in the category $\mathcal{E}$ from the object $%
(S,U_{0}(S))$ to the object $(T,U_{0}(T)).$ Thus, $U_{0}$ can be seen as a
functor from the category $\mathcal{CR}^{s}$ of commutative rings and
surjective homomorphisms to the category $\mathcal{E}.$

\begin{figure}[!htb]
\centering
\begin{tikzpicture}[scale=0.8]
	\draw[->] (3,-0.3) --node[left]{\footnotesize $h$} (3,-1.7);
	\draw[->] (6.8,-0.3) --node[right]{\footnotesize $U_{0}(h)$} (6.8,-1.7);
    \draw[->] (4,1) --node[above]{\scriptsize $U_{0}$} (5.8,1);
    \draw[|->] (4.3,-1.1) -- (5.5,-1.1);
	\draw (2.8,1) node{\small $\mathcal{CR}^{s}$} (7,1) node{\small $\mathcal{E}$};
	\draw (3,0) node{\small $S$} (3,-2.2) node{\small $T$};
	\draw (7,0) node{\small $(S,U_{0}(S))$} (7,-2.2) node{\small $(T,U_{0}(T))$};
\end{tikzpicture}
\end{figure}

The universal property of $U_{0}$ allows us to conclude immediately the
following theorem:

\begin{theorem}
For each ring $S,$ $U_{0}(S)$ is an initial object of the category $\mathcal{%
E}(S).$
\end{theorem}

\begin{proof}
Let $R$ be an object of $\mathcal{E}(S)$ and let $u:S\rightarrow R$ be the
inclusion homomorphism. The universal property guarantees that there is a 
unique homomorphism of unitary rings $u_{R}:U_{0}(S)\rightarrow R$ such
that $u_{R}\circ \iota _{S}=u,$ that is, $u_{R}$ restricted to $S$ is the
identity.
\end{proof}

As the functor $\mathrm{NC}_{S}$ is injective in objects, it is clear that
the image of $\mathrm{NC}_{S}$ is a subcategory of $\mathbb{S}$, that we
denote $\mathcal{NC}(S)$ (category of nilcompactifications of $S$). We
denote $\mathrm{NC}_{0}$ the functor $\mathrm{NC}\circ U_{0}:\mathcal{CR}%
^{s}\rightarrow \mathbb{S}.$ Thus, we have the following result:

\begin{corollary}
For each ring $S,$ $\mathrm{NC}_{0}\left( S\right) $ is a final object of
the category $\mathcal{NC}(S).$
\end{corollary}

\begin{remark}
As a consequence of Corollary \ref{Surjec hom then NC homeomo} we have that $%
\mathrm{NC}_{0}\left( S\right) $ is homeomorphic to $\mathrm{NC}_{0}\left(
S/N(S)\right) .$ Therefore we can reduce the study of nilcompactifications
to semi-prime or reduced rings.
\end{remark}

\section{Third variation}

In this section we are working again with a fixed ring $S$. Let $R$ be an
object of $\mathcal{E}(S)$ and consider the unique morphism $%
u_{R}:U_{0}(S)\rightarrow R$ of $\mathcal{E}(S).$ Then $\mathrm{NC}%
_{S}(u_{R}):\mathrm{NC}_{S}(R)\rightarrow \mathrm{NC}_{0}\left( S\right) $
is a (strongly) continuous function. We denote $\eta (R)$ the image of the
function $\mathrm{NC}_{S}(u_{R}).$

\begin{theorem}
If $R$ is an object of $\mathcal{E}(S)$ then $\eta (R)$ is an A class
compactification of $\mathrm{Spec}S$ and besides it is dense in $\mathrm{NC}%
_{0}\left( S\right) .$
\end{theorem}

\begin{proof}
It is enough to observe that $\mathrm{Spec}S\subseteq \eta (R)\subseteq \mathrm{NC}%
_{0}\left( S\right) $ and that $\mathrm{Spec}S$ is an open dense subspace of $\mathrm{%
NC}_{0}\left( S\right) .$
\end{proof}

We consider the pre-order between compactifications given in \cite{Margalef}%
: let $\left( X^{\prime },\tau ^{\prime }\right) $ and $\left( X^{\prime
\prime },\tau ^{\prime \prime }\right) $ be compactifications of $\left(
X,\tau \right) ,$ with immersions $f^{\prime }:\left( X,\tau \right)
\rightarrow \left( X^{\prime },\tau ^{\prime }\right) $ and $f^{\prime
\prime }:\left( X,\tau \right) \rightarrow \left( X^{\prime \prime },\tau
^{\prime \prime }\right) .$ It is said that $\left( X^{\prime },\tau
^{\prime }\right) $ $\preceq \left( X^{\prime \prime },\tau ^{\prime \prime
}\right) $ if there exists $h:\left( X^{\prime \prime },\tau ^{\prime \prime
}\right) \longrightarrow \left( X^{\prime },\tau ^{\prime }\right) $
continuous and surjective, such that $f^{\prime }=h\circ f^{\prime \prime }$%
. We obtain directly the following result:

\begin{proposition}
If $R$ is an object of $\mathcal{E}(S)$ then $\eta (R)$ is a
compactification of $\mathrm{Spec}S$ smaller than $\mathrm{NC}_{S}(R).$
\end{proposition}

\begin{corollary}
If $\mathrm{NC}_{0}\left( S\right) $ is a Hausdorff space then, it is the
smallest nilcompactification of $\mathrm{Spec}S.$
\end{corollary}

\begin{proof}
If $R$ is an object of $\mathcal{E}(S)$ then $\eta (R)$ is a compact subset of $\mathrm{NC}_{0}\left( S\right) $, then it is closed. Besides, $\eta (R)$ is dense in $\mathrm{NC}_{0}\left( S\right) $, therefore $\mathrm{NC}_{S}(u_{R})$ is surjective.
\end{proof}

The proof of the following proposition is a simple routinary exercise:

\begin{proposition}
If $h:R\rightarrow T$ is a morphism of the category $\mathcal{E}(S)$ then $%
\eta (T)\subseteq \eta (R).$
\end{proposition}

If for each morphism $h:R\rightarrow T$ of the category $\mathcal{E}(S)$ we
define $\eta (h):\eta (T)\rightarrow \eta (R):x\mapsto x$ then, $\eta $ is a
functor from the category $\mathcal{E}(S)$ to the category $\mathrm{Top}$ of
topological spaces and continuous functions.

We have then the following natural question: if $R$ is an object of $%
\mathcal{E}(S)$, is $\eta (R)$ a spectral space?

\section{Two examples}

In this section we present two examples which illustrate some results of
this paper.

\begin{example}
Let $B$ be a Boolean ring without identity. As $B$ is semiprime, $N\left(
B\right) =0$. We are going to find $\psi _{\left( B,U_{0}\left( B\right)
\right) }\left( 0\right) :$

\noindent $\left( a,\alpha \right) \in \psi _{\left( B,U_{0}\left( B\right)
\right) }\left( 0\right) $ if and only if $ax+\alpha x=0$ for each $x\in B,$
that is, $ax=\alpha x$ for each $x\in B.$

\noindent If $\alpha $ is even then $ax=0$ for all $x\in B,$ from which it
follows that $a=0.$

\noindent If $\alpha $ is odd then $ax=x$ for all $x\in B$ and therefore, $a$
is the identity of $B,$ which is absurd.

\noindent Hence, $\psi _{\left( B,U_{0}\left( B\right) \right) }\left(
0\right) =\left\{ 0\right\} \times 2\mathbb{Z}$ and $Q\left( B,U_{0}\left(
B\right) \right) =U_{0}\left( B\right) /\left( \left\{ 0\right\} \times 2%
\mathbb{Z}\right) .$

\noindent In this case, $NC\left( B,U_{0}\left( B\right) \right) $ is
precisely the Alexandroff compactification of $\mathrm{Spec}B$ (see \cite%
{Acosta-Galeano}).
\end{example}

The following example shows that different i-extensions of a ring can
produce different nilcompactifications:

\begin{example}
Consider the ring without identity $S=x\mathbb{R}\left[ x\right] $. Two
different unitary i-ex\-ten\-sio\-ns of $S$ are $U_{0}\left( S\right) $ and $%
\mathbb{R}\left[ x\right] .$ Notice that the unique homomorphism from $%
U_{0}\left( S\right) $ to $\mathbb{R}\left[ x\right] $ is not surjective and
there not exists a homomorphism from $\mathbb{R}\left[ x\right] $ to $%
U_{0}\left( S\right) .$ It is easy to see that $\psi _{\left( S,U_{0}\left(
S\right) \right) }\left( 0\right) =0$ and $\psi _{\left( S,\mathbb{R}\left[ x%
\right] \right) }\left( 0\right) =0.$ Therefore, $Q\left( S,U_{0}\left(
S\right) \right) =U_{0}\left( S\right) $ and $Q\left( S,\mathbb{R}\left[ x%
\right] \right) =\mathbb{R}\left[ x\right] ,$ thus $NC\left( S,U_{0}\left(
S\right) \right) =\mathrm{Spec}\left( U_{0}\left( S\right) \right) $ and $%
NC\left( S,\mathbb{R}\left[ x\right] \right) =\mathrm{Spec}\mathbb{R}\left[ x%
\right] $. If we consider the surjective homomorphisms%
\begin{eqnarray*}
\pi &:&U_{0}\left( S\right) \rightarrow \mathbb{Z}:\left( s,z\right) \mapsto
z, \\
\beta &:&\mathbb{R}\left[ x\right] \rightarrow \mathbb{R}:p\left( x\right)
\mapsto p\left( 0\right) ,
\end{eqnarray*}%
we conclude, by the Correspondence Theorem, that $NC\left( S,U_{0}\left(
S\right) \right) $ is a compactification of $\mathrm{Spec}S$ by enumerable
points, while $NC\left( S,\mathbb{R}\left[ x\right] \right) $ is a
compactification of $\mathrm{Spec}S$ by one point.
\end{example}

The following diagram summarizes the ideas presented in this paper.

\begin{figure}[!htb]
\centering
\begin{tikzpicture}[scale=0.5]
    \draw (0,0) rectangle (2,2); \draw (1,1) node{$\mathcal{E}$};
    \draw (14,0) rectangle (16,2); \draw (15,1) node{$\mathbb{S}$};
    \draw (14,-6) rectangle (16,-4); \draw (15,-5) node{$\mathcal{CR}^{s}$};
    \draw (14,6) rectangle (16,8); \draw (15,7) node{$\mathcal{CR}_{1}$};
    \draw (5,6) rectangle (7,8); \draw (6,7) node{$\mathcal{CR}$};
    \draw (0,13) rectangle (2,15); \draw (1,14) node{$\mathcal{E}(S)$};
    \draw (19,13) rectangle (21,15); \draw (20,14) node{$\mathrm{Top}$};
	\draw[->] (3,0) -- (13,0); \draw (5,0.3) node{\footnotesize $\mathrm{Spec}\circ V$};
	\draw[->] (3,1) -- (13,1); \draw (5,1.3) node{\footnotesize $\mathrm{Spec}\circ W$};
	\draw[->] (3,2) -- (13,2); \draw (10.3,2.3) node{\footnotesize $NC=\mathrm{Spec}\circ Q$};
	\draw[double,->] (8.2,0.8) --node[right]{\footnotesize $\psi $} (8.2,0.2); 
	\draw[double,->] (8.2,1.2) --node[right]{\footnotesize $\lambda $} (8.2,1.8); 
	\draw[->] (2,-1) -- (13,-4.5); \draw (10,-3.1) node{\scriptsize $W$};
	\draw[->] (13,-6) -- (1,-2); \draw (5.2,-3.9) node{\scriptsize $U_{0}$};
	\draw (7.8,-3.5) node{$\dashv$};
	\draw[->] (14.5,-3.5) -- (14.5,-0.5); \draw (13.5,-2) node{\footnotesize $\mathrm{Spec}$};
	\draw[->] (15.5,-3.5) -- (15.5,-0.5); \draw (18,-2) node{\footnotesize $\mathrm{NC}_{0}=\mathrm{NC}\circ U_{0}$};
	\draw (16.7,-5.3) rectangle (24.85,-3.5);
	\draw (20.7,-4) node{\scriptsize $\implies$ : natural transformation};
	\draw (17.5,-4.8) node{\scriptsize $\dashv $}; \draw (20.55,-4.8) node{\scriptsize : is left adjoint of};
	\draw[->] (0,12.5) -- (0,2.5); \draw (0.4,8.5) node{\scriptsize $\subseteq$};
	\draw[->] (1.5,3) -- (4.5,5.8); \draw (3.8,4.5) node{\scriptsize $V$};
	\draw[->] (0.5,3.2) -- (4.2,6.9); \draw (2.4,5.8) node{\scriptsize $\chi$};
	\draw[double,->] (2.7,5.2) -- (3.1,4.8); \draw (3.3,5.4) node{\scriptsize $j$};
	\draw[->] (2,3) -- (13.2,6.5); \draw (8.5,5.4) node{\scriptsize $V$};
	\draw[->] (3,2.5) -- (13.5,5.8); \draw (9,3.9) node{\scriptsize $Q$};
	\draw[double,->] (8.6,4.9) -- (8.8,4.4); \draw (9.2,4.9) node{\scriptsize $\theta$};
	\draw[->] (15,5.5) -- (15,2.5); \draw (15.9,3.8) node{\footnotesize $\mathrm{Spec}$};
	\draw[->] (13,7) --node[above]{\footnotesize $\subseteq$} (8,7);
    \draw[->] (3,15) --node[above]{\footnotesize $\eta$} (18,15);
    \draw[->] (3,14) --node[below]{\footnotesize $\mathrm{NC}_{S}$} (18,14);
    \draw[double,->] (10.5,14.8) --node[right]{\scriptsize $\subseteq$} (10.5,14.2);
    \draw (16.5,0) -- (21,0);
    \draw[->] (21,0) --node[right]{\scriptsize $\subseteq$} (21,12.5);
    \draw[dashed] (3,13.1) -- (19,9) --node[right]{\footnotesize $\eta ?$} (19,1);
    \draw[dashed][->] (19,1) -- (16.5,1);
    \draw (2.5,12.7) -- (18,8.5) --node[left]{\footnotesize $\mathrm{NC}_{S}$} (18,2);
    \draw[->] (18,2) -- (16.5,2);
    \draw[->] (1,12) --node[below]{\footnotesize $\psi _{S}$} (13.3,8.2);
\end{tikzpicture}
\end{figure}

\end{document}